\documentclass[11pt]{article}
\usepackage{latexsym}
\usepackage{graphicx} 
\usepackage[all]{xypic}
\usepackage{color} 
\usepackage{amsmath, amsthm, amssymb}
\usepackage{enumerate}
\usepackage{float}
\usepackage{chemarr}
\usepackage{url}
\usepackage{stackrel}
\usepackage[small,bf]{caption}
\usepackage{wrapfig}

\newtheorem{theorem}{Theorem}[section]

\newtheorem{lemma}[theorem]{Lemma}
\newtheorem{example}[theorem]{Example}

\newtheorem{definition}[theorem]{Definition}

\newtheorem{remark}[theorem]{Remark}

\newcommand{\be}{\begin{equation}}
\newcommand{\ee}{\end{equation}}
\newcommand{\bd}{\begin{displaymath}}
\newcommand{\ed}{\end{displaymath}}
\newcommand{\beal}{\begin{align}}
\newcommand{\enal}{\end{align}}
\newcommand{\been}{\begin{enumerate}}
\newcommand{\enen}{\end{enumerate}}
\newcommand{\beit}{\begin{itemize}}
\newcommand{\enit}{\end{itemize}}

\newcommand{\CRN}{chemical reaction network }
\newcommand{\CRNNoSpace}{chemical reaction network}

\newcommand{\doalg}{Deficiency One Algorithm }
\newcommand{\doalgnospace}{Deficiency One Algorithm}

\newcommand{\Net}{\mathfrak{G}}
\newcommand{\invtPoly}{\mathcal{P}} 

\newcommand{\SSS}{\mathcal S}
\newcommand{\CC}{\mathcal C}
\newcommand{\RR}{\mathcal R}

\providecommand{\abs}[1]{\lvert#1\rvert}

\def\ds{\displaystyle}
\def\de{\delta}

\newcommand{\ra}{\rightarrow}
\newcommand{\lra}{\leftrightarrows}
\newcommand{\rla}{\rightleftarrows}

\newcommand{\R}{\mathbb{R}}
\newcommand{\Rplus}{\mathbb{R}_{>0}}

\newcommand{\Z}{\mathbb{Z}}

\newcommand{\im}{\operatorname{Im}}

\begin{document}

\title{Complete characterization by multistationarity of fully open networks with one non-flow reaction}


\author{Badal Joshi}


\date{}

\maketitle

\begin{abstract}
This article characterizes certain small multistationary chemical reaction networks. We consider the set of fully open networks, those for which all chemical species participate in inflow and outflow, containing one non-flow (reversible or irreversible) reaction. We show that such a network admits multiple positive mass-action steady states if and only if the stoichiometric coefficients in the non-flow reaction satisfy a certain simple arithmetic relation. The multistationary fully open one-reaction networks are identified with the chemical process of autocatalysis. Using the notion of `embedded network' defined recently by Joshi and Shiu, we provide new sufficient conditions for establishing multistationarity of fully open networks, applicable well beyond the one-reaction setting.
\end{abstract}

\section{Introduction}
Chemical reaction networks are used to model systems that occur in chemical engineering and systems biology. The property of existence of multiple positive steady states (also known as multistationarity) provides the mathematical underpinnings for a biochemical network to act as a switch \cite{markevich2004signaling, thomson2009unlimited}. Therefore it is an important problem to determine which chemical reaction networks admit multiple positive steady states. Determining whether a chemical reaction network admits multiple positive steady states is difficult: for instance, in the mass-action kinetics setting, it requires determining existence of multiple positive solutions to a system of multivariate polynomials with unknown coefficients. Several criteria exist which help rule out multistationarity in chemical reaction networks. Important examples of such criteria are Deficiency Zero and Deficiency One Theorems of Horn, Jackson, and Feinberg \cite{feinberg1972complex, FeinDefZeroOne, FeinbergMSSdefone, horn1972necessary, HornJackson72}, the Jacobian criterion and the more general injectivity test of Craciun and Feinberg \cite{ME1,ME_entrapped,ME2,ME3,JiThesis};  the graphical criteria of Soul\'e \cite{soule2004graphic} and in more general settings than mass-action, the work of Banaji {\it et al.} \cite{BanajiCraciun2010, BanajiDonBai}; also see the more recent extensions by Feliu and Wiuf \cite{feliu2012preclusion}, Joshi and Shiu \cite{joshi2012simplifying}, and Gnacadja \cite{gnacadja2012jacobian}. The results in these papers provide sufficient conditions for ruling out multistationarity, which alternatively may be viewed as providing necessary conditions for establishing multistationarity since avoiding the conditions is necessary for multiple steady states. On the other hand, sufficient conditions for establishing multistationarity are relatively rare. Instances where multistationarity can be established include Feinberg's Deficiency One Algorithm \cite{FeinbergMSSdefone} and Ellison and Feinberg's Advanced Deficiency Algorithm \cite{EllisonThesis}. These criteria have been implemented in the Chemical Reaction Network Toolbox, software available online for free download and use \cite{Toolbox}. Other results on establishing multistationarity include the work of Conradi {\it et al.} \cite{conradi2012switching, conradi2007subnetwork}. 

Within the fully open network setting (a fully open network is a chemical reaction network where all chemical species participate in inflow and outflow), recent results by Joshi and Shiu \cite{joshi2012atoms} give a new approach for establishing multistationarity via `atoms of multistationarity' (see Definition \ref{def:atom}). Possessing an atom of multistationarity as an `embedded network' (see Definition \ref{def:embedded_network}) is a sufficient condition for multistationarity in fully open networks. Using this approach, the problem of classifying fully open networks by multistationarity may be reduced to two relatively simpler problems: 1) determining the atoms of multistationarity,  and 2) determining whether a network possesses one of the known atoms of multistationarity as an embedded network. Here we focus on the first problem, and provide an answer for the smallest networks. The next examples illustrate the type of questions that the results in this article will enable us to answer.

\begin{example} \label{ex:3nets}
Consider the following fully open networks N1-N3 in species $A,B,C,D$ and $E$. By the network property of being `fully open' we mean that the flow reactions $0 \rla A$, $0 \rla B$, $0 \rla C$, $0 \rla D$, and $0 \rla E$ are included in all three networks. 
\been[N1:]
\item $\qquad A+B \ra A+C \qquad 2B \ra A+D \qquad A+2E \ra 3E$. 
\item $\qquad A+B \ra A+C \qquad 2B \ra A+D \qquad A+E \ra 2E$. 
\item $\qquad A+C \ra A+B \qquad 2B \ra A+D \qquad A+E \ra 2E$. 
\enen  
{\em Does the network N1 (or N2, or N3), when endowed with mass-action kinetics (see Definition \ref{def:mass-action}), admit a choice of positive rate constants for which N1 (or N2, or N3) has multiple positive steady states?} Note that N2 differs from N1 only in the third reaction and only in the stoichiometric coefficients of the species $E$. Moreover, N2 differs from N3 only in the direction of the first reaction. We will demonstrate the delicate dependence on the network structure by showing that only N1 and N3 admit multiple positive steady states -- both by virtue of possessing known atoms of multistationarity, while N2 does not admit multiple steady states.
\end{example}

\begin{example} \label{ex:schlosser}
Consider the following fully open networks M1-M3 appearing in the work of Schlosser and Feinberg \cite{schlosser1994theory}. All networks have the `fully open' property of having all chemical species participate in the inflow and the outflow.
\been[M1:]
\item $A+B \lra 2A$. 
\item $2A+B \lra 3A$.
\item $A+2B \lra 3A$.
\enen
The main theorem (Theorem \ref{thm:1rxn}) in this article will show that in the mass-action setting, only network M2 has the capacity for multiple positive steady states, while the networks M1 and M3 cannot admit multiple positive steady states no matter what positive reaction rate constants are chosen. 
\end{example}


In this work, we characterize the class of `smallest' atoms of multistationarity, namely those containing one non-flow reaction, which may be irreversible or reversible. This is a continuation of the work in \cite{joshi2012atoms}, where the authors catalog all two-reaction bimolecular atoms of multistationarity.  Atoms of multistationarity containing one non-flow reaction will be referred to as {\em one-reaction atoms of multistationarity}. Consider the following general one-reaction fully open network consisting of $s$ species all of which are in the inflow and outflow:
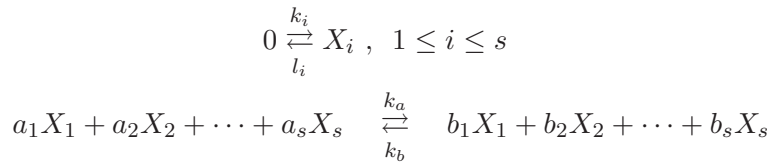
\begin{align*}
0   \stackrel[l_i]{k_i}{\rightleftarrows} X_i ~,& ~~ 1 \le i \le s\\
a_1 X_1 + a_2 X_2 + \cdots + a_s X_s \quad &\stackrel[k_b]{k_a}{\rightleftarrows}  \quad b_1 X_1 + b_2 X_2 + \cdots + b_s X_s
\end{align*}
where at least one of the rate constants $k_a$ or $k_b$ is assumed to be positive. The $k_i$ and $l_i$ are positive rate constants which denote the rate at which the species $X_i$ flows in and out, respectively. The stoichiometric coefficients $a_i$ and $b_i$ are assumed to be non-negative integers. The main theorem (Theorem \ref{thm:1rxn}) in this article gives a simple arithmetic relation on the stoichiometric coefficients which establishes whether the network is multistationary or not. 

Two important results follow from Theorem \ref{thm:1rxn}. The first result is Theorem \ref{thm:1rxnatoms}, which gives a classification of the entire set of one-reaction atoms of multistationarity. We find that the infinitely many one-reaction atoms of multistationarity can be classified into two types, each type parametrized by two integers. The first type contains one chemical species and the second type contains two chemical species. Furthermore, the non-flow reaction in both types of atoms is irreversible. As corollaries of Theorem \ref{thm:1rxnatoms}, we find that: 1) there are no one-reaction atoms of multistationarity with a reversible non-flow reaction (in other words, if a one-reaction network with a reversible non-flow reaction is multistationary, then it contains a multistationary subnetwork which is fully open and has an irreversible non-flow reaction) and 2) a bimolecular reaction network containing one non-flow reaction (which may be reversible or irreversible) does not admit multiple steady states. The second result that follows from Theorem \ref{thm:1rxn} is Theorem \ref{thm:largenets}, which is obtained by combining Theorem \ref{thm:1rxn} with the `embedded network theorem' of Joshi and Shiu  \cite{joshi2012atoms} and extends the applicability of Theorem \ref{thm:1rxn} beyond the setting of one-reaction networks. Theorem \ref{thm:largenets} states that a fully open network with any number of non-flow reactions admits multiple steady states if it possesses a one-reaction atom of multistationarity as an embedded network. 

We find that the multistationary one-reaction fully open networks including the one-reaction atoms of multistationarity are identified with the chemical process of autocatalysis. More precisely, a one-reaction fully open network is multistationary if and only if the network contains a non-flow reaction with a set of species that are autocatalytic ({\it i.e.} they appear with a higher stoichiometric coefficient in the product complex than in the reactant complex), and the sum of the stoichiometric coefficients of such autocatalytic species in the reactant complex is at least two. 

Other authors have previously approached the problem of identifying the smallest chemical reaction networks with a certain specified property. Smallest multistationary chemical reaction networks with the mass-preserving property have been studied in \cite{Smallest}. The smallest chemical reaction outside the fully open network setting (smallest by number of species, number of reactions, and number of terms in the differential equation)  was studied in \cite{ThomasSmallest} and the smallest chemical reaction network with Hopf bifurcation was studied in \cite{smallestHopf,smallestHopf2}. Other examples of classification by multistationarity of small networks include \cite{feliu2012enzyme,SmallGRN,Emergence}. Recently, generalized catalytic and autocatalytic networks have been studied in \cite{manojcatalysis}. 

This article is organized as follows. Section 2 provides the background information on chemical reaction networks including the basic definitions, notation and the Deficiency One Theorem of Feinberg. Section 3 provides a review of the Deficiency One Algorithm of Feinberg. In Section 4, we state and prove our main theorem which gives a characterization of one-reaction fully open networks by multistationarity. The first corollary gives a complete classification of one-reaction atoms of multistationarity. As a second corollary we get sufficient conditions for establishing multistationarity of larger networks by way of one-reaction atoms of multistationarity. 


\section{Chemical reaction network theory} \label{sec:intro_CRN}
We begin with a review of the notation and basic definitions related to chemical reaction networks. An example of a {\em chemical reaction} is the following:
\begin{align} \label{eq:exreaction}
  X_{1}+ 2X_{2} ~\rightarrow~ X_2 + X_{3} ~.
\end{align}
The $X_{i}$ are called chemical {\em species}, and $X_{1}+2X_{2}$ and
$X_2 + X_{3}$ are called chemical {\em complexes.}  For the reaction in \eqref{eq:exreaction}, $y := X_{1}+2X_{2}$ is called the {\em reactant complex} and $y' := X_2 + X_{3}$ is called the product complex, so we may rewrite the reaction as $y \ra y'$. We will often find it convenient to think of the complexes as vectors, for instance, we may assign the reactant complex $X_{1}+2X_{2}$ to the vector $(1,2,0)$ and the product complex  $X_2 + X_{3}$ to the vector $(0,1,1)$. In other words, we are identifying the species $X_i$ with the canonical basis vector whose $i$-th component is $1$ and the other components are $0$. We let $s$ denote the total number of species $X_i$ and we consider a set of $r$ reactions, each denoted by  $y_{k} \rightarrow y_{k}'$, for $k \in \{1,2,\dots,r\}$, and $y_k, y_k' \in \Z^s_{\ge 0}$, with $y_k \ne y_k'$. We index the entries of a 
complex vector $y_k$ by writing $y_k = \left( y_{k1}, y_{k2}, \dots, y_{ks}\right) \in \Z^s_{\geq 0}$,
and we will call $y_{ki}$ the {\em stoichiometric coefficient} of species $i$ in 
complex $y_k$.  
For ease of notation, when there is no need for
enumeration we typically will drop the subscript $k$ from the notation
for the complexes and reactions.

Many of the definitions in this paper follow those in Joshi and Shiu \cite{joshi2012atoms}; we start by defining chemical reaction networks.

\begin{definition}   \label{def:crn}
  Let $\SSS = \{X_i\}$, $\CC = \{ y\},$ and $\RR = \{ y \to  y' | y' \ne y\}$ denote finite sets of species, complexes, and reactions, respectively.  The triple
  $\{\SSS, \CC, \RR \}$ is called a {\em chemical reaction network} if it satisfies the following:
  \been
  	\item for each complex $ y \in \CC$, there exists a reaction in $\RR$ for which $ y$ is the reactant complex or $ y$ is the product complex, and
	\item for each species $X_i \in \SSS$, there exists a complex $ y \in \CC$ that contains $X_i$.
  \enen
\end{definition}

For a chemical reaction network $\{\SSS, \CC, \RR \}$, unless otherwise specified, we will denote the number of species by $s := \abs{\SSS}$, the number of complexes by $n :=\abs{\CC}$ and the number of reactions by $r:= \abs{\RR}$. 

A subset of the reactions $\RR' \subset \RR$ defines the {\em subnetwork} $\{\SSS|_{\CC|_{\RR'}},\CC|_{\RR'},\RR' \}$, where $\CC|_{\RR'}$ denotes the set of complexes that appear in the reactions $\RR'$, and $\SSS|_{\CC|_{\RR'}}$ denotes the set of species that appear in those complexes. We now define the notion of an embedded network, a more general notion than a subnetwork. 

\begin{definition} \label{def:embedded_network}
Let $\Net=\{\SSS,\CC,\RR\}$ be a \CRNNoSpace.
\begin{enumerate}
	\item Consider a subset of the species $S\subset \SSS$, a subset of the complexes $C \subset \CC$, and a subset of the reactions $R\subset \RR$. 
	\beit
		\item 
		The {\em restriction of $R$ to $S$}, denoted by $R|_S$, is the set of reactions obtained by taking the reactions in $R$ and removing all species not in $S$ from the reactant and product complexes. If a reactant or a product complex does not contain any species from the set $S$, then the complex is replaced by the $0$ complex in $R|_S$. If a trivial reaction (one in which the reactant and product complexes are the same) is obtained in this process, then that reaction is removed.  Also removed are extra copies of repeated reactions.  
		\item The {\em restriction of $C$ to $R$}, denoted by $C|_R$, is the set of (reactant and product) complexes of the reactions in $R$.  
		\item The {\em restriction of $S$ to $C$}, denoted by $S|_C$, is the set of species that are in the complexes in $C$. 
	\enit
	
	\item The network obtained from $\Net$ by {\em removing a set of reactions} $\{y \ra y'\} \subset \RR$ is the subnetwork 
		$$\left\{\SSS|_{\CC|_{\RR \setminus \{y \ra y' \}}},~\CC|_{\RR \setminus \{y \ra y' \}},~\RR \setminus \{y \ra y' \} \right\}~.$$

	\item The network obtained from $\Net$ by {\em removing a subset of species} $\{X_i\} \subset \SSS$ is the network 
		$$\left\{\SSS |_{\CC|_{\RR|_{\SSS\setminus \{X_i\}}}} ,~\CC|_{\RR|_{\SSS\setminus \{X_i\}}},~
			\RR|_{\SSS\setminus \{X_i\}}  \right\}~.$$
	\item
	Let $\Net=\{\SSS,\CC,\RR\}$ be a \CRNNoSpace.
	An {\em embedded network} of $\Net$, which is defined by a subset of the reactions, $R \subset \RR$, and a subset of the species, $S \subset \SSS$, where $S$ has the property that $S = S|_{\CC|_{R|_S}}$
	is the network $\{S,\CC|_{R|_S},R|_S\}$ consisting of the reactions $R|_S$. 
\end{enumerate}
\end{definition}

\begin{example}
We demonstrate the operations of removing reactions, and of removing species by considering the example of the following chemical reaction network:
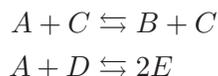
\begin{align*}
A+C &\lra B+C \\
A+D &\lra 2E
\end{align*}
For this network $\SSS = \{A, B, C, D, E\}$, $\CC = \{A+C, B+C,A+D,2E\}$, $\RR=\{A+C \lra B+C, A+D \lra 2E\}$. 
\been
\item Consider the operation of removing the second reversible reaction; let $R = \{A+C \lra B+C\} \subset \RR$. So that $\CC |_R = \{ A+C, B+C\}$ and $\SSS|_{\CC |_R} = \{A,B,C \}$. Thus, removing the reactions in $\RR \setminus R$ results in the subnetwork $\{ \{A,B,C \}, \{ A+C, B+C\} ,\{A+C \lra B+C\}\}$. 
\item Now consider the operation of removing the species $A$ and $B$. Let $S = \SSS \setminus \{A,B\} = \{C,D,E\}$. In this case, $\RR|_S = \{ D \lra 2E\}$, $\CC|_{\RR|_S} = \{D, 2E\}$, and $\SSS|_{\CC|_{\RR|_S}} = \{ D, E\}$. Thus, removing species $A$ and $B$ results in the embedded subnetwork $\{ \{ D, E\}, \{D, 2E\},  \{ D \lra 2E\}\}$. In general, $\SSS|_{\CC|_{\RR|_S}}$ is a subset of $S$ but may not be equal to $S$, as this example illustrates. 
\enen
\end{example}

%

This article studies a class of chemical reaction networks called fully open networks, also referred to as fully open CFSTRs in the literature (see for instance \cite{ME1,joshi2012atoms}). 

\begin{definition} \label{def:CFSTR}
\begin{enumerate}
  	\item 
	A {\em flow reaction} contains only one species with the sum of the stoichiometric coefficient in the reactant and product complex for that species equal to one. In more concrete terms, a flow reaction can be either an {\em inflow reaction} $ 0 \ra X_i$ or an {\em outflow reaction} $X_i \ra  0$.  A {\em non-flow reaction} is any reaction that is not a flow reaction.
	\item   
  A \CRN is a {\em continuous-flow stirred-tank reactor} {\em (CFSTR)} if it contains all outflow reactions $X_i \ra  0$ (for all $X_i \in \SSS$) and a CFSTR is a {\em fully open network} if it contains all inflow reactions $ 0 \ra X_i$. We note that a fully open network is referred to as a fully open CFSTR in \cite{joshi2012atoms}. 
  \item A {\em one-reaction fully open network} is a fully open network with either one irreversible non-flow reaction or one reversible non-flow reaction.
  \item An {\em autocatalytic reaction} is a chemical reaction in which at least one chemical species $X_i$ appears in both the product and the reactant complex and the stoichiometric coefficient of $X_i$ is higher in the product complex than in the reactant complex. In this case, the species $X_i$ is referred to as an {\em autocatalytic species}. For instance, $A+B \ra 2A+C$ is an autocatalytic reaction and $A$ is an autocatalytic species.
  \end{enumerate}
\end{definition}
\noindent

We define the function {\em sign} on real numbers by 
\begin{align*}
\text{sign} (x) = \begin{cases}1 & \text{if } x>0 \\0 & \text{if } x=0 \\-1 & \text{if } x<0  \end{cases}
\end{align*}

\begin{definition}
\been
\item The {\em stoichiometric subspace} of a network is the vector space spanned by the reaction vectors of the network, $\mathbb S:=\text{span} (\{ y' -  y |  y \ra  y' \in \RR \})$. 
\item Two vectors $\alpha, \beta \in \R^s$ are said to be {\em stoichiometrically compatible} if $\beta - \alpha \in \mathbb S$. 
\item A set of all stoichiometrically compatible positive vectors forms a {\em positive stoichiometric compatibility class}. For instance, the positive stoichiometric compatibility class containing the vector $\alpha \in \R^{s}_{> 0}$ is  $\invtPoly_\alpha =  (\alpha+ \mathbb S) \cap \mathbb{R}^{s}_{\geq 0}$. 
\item A vector $ \mu \in \R^{s}$ is said to be {\em sign compatible with the stoichiometric subspace} if there exists a $\lambda \in \mathbb S$ such that $\text{sign} (\mu_i) = \text{sign} (\lambda_i)$, for $1 \le i \le s$.  

\enen
\end{definition}

\begin{definition} \label{def:mass-action}
Let $x_i$ represent the concentration of the chemical species $X_i$ and let $ x = (x_1,\ldots,x_{s})$. Let $ y_k = (y_{k1}, \ldots, y_{ks})$ be the stoichiometric coefficients of the species $(X_1,\ldots, X_{s})$ in the reactant complex of the $k$-th reaction. For a vector of positive {\em reaction rate constants} $(\kappa_1, \kappa_2, \dots, \kappa_{r}) \in \Rplus^{r} $, let the reaction rate be given by:
\begin{equation}
\kappa_k  x^{ y_k} :=  \kappa_k x_1^{y_{k1}} x_2^{y_{k2}} \cdots x_{s}^{y_{k{s}}}
  \label{eq:massaction}
\end{equation}
where by convention $0^0=1$. 

A chemical reaction network $\Net$ is said to be endowed with {\em mass-action kinetics} if for a specified set of positive reaction rate constants $(\kappa_1, \kappa_2, \dots, \kappa_{r}) \in \Rplus^{r}$, $x(t)$ is governed by the following ODEs (which we will refer to as {\em mass-action ODEs}): 
\begin{equation}
		  \dot { x}(t) \quad = \quad \sum_{k=1}^{r } \kappa_k  x(t)^{ y_k}( y_k' -  y_k) \quad =: \quad  f( x(t))~.
  \label{eq:main}
\end{equation}
\end{definition}

\begin{definition}
\been
\item A concentration vector $x \in \R^{s}_{> 0}$ is a (positive) {\em steady state} of the system~\eqref{eq:main} if $ f(x) = 0$. 
\item A steady state $x$ is {\em nondegenerate} if $\im d f (x) = \mathbb S$.
	(Here, ``$d f (x)$'' is the Jacobian matrix of $ f$ at $x$: the ${s} \times {s}$-matrix whose $(i,j)$-th entry is equal to the partial derivative $\frac{\partial f_i}{\partial x_j} (x)$).
\item A chemical reaction network  $\Net$ endowed with mass-action kinetics is said to {\em admit multiple steady states} if there exist positive reaction rate constants $(\kappa_1, \kappa_2, \dots, \kappa_{r}) \in \Rplus^{r} $ such that the resulting system of mass-action ODEs, $\dot x(t) = f(x(t))$ has multiple positive steady states within some positive stoichiometric compatibility class. If a network $\Net$ admits multiple steady states (MSS) we will say that $\Net$ is {\em multistationary}. 
\item If a network $\Net$ admits MSS and these steady states are nondegenerate, we will say that $\Net$ admits nondegenerate MSS.  
	\enen
\end{definition}

Note that~\eqref{eq:main} implies that a trajectory $x(t)$ that begins at a positive vector $x(0) \in \R^s_{>0}$ remains in the stoichiometric compatibility class containing $x(0)$,  $\invtPoly_{x(0)} = (x(0)+\mathbb S) \cap \mathbb{R}^{s}_{\geq 0}$ for all positive time. In other words, $\invtPoly_{x(0)}$ is forward-invariant with respect to~\eqref{eq:main} \cite{volpert1975analysis}. 

In the case of a fully open network, the reaction vector for the $i$-th inflow reaction is the $i$-th canonical basis vector of $\R^{s}$, so the stoichiometric subspace is $\mathbb S = \R^{s}$.  It follows that for a fully open network, the unique positive stoichiometric compatibility class for all positive vectors $x \in \R^s_{> 0}$ is the nonnegative orthant: $\invtPoly_x=\R^{s}_{\geq 0}$.

A chemical reaction network can be viewed as a graph whose vertex set is $\CC$ and whose edge set is $\RR$. The next few definitions address the graph-related structure of a chemical reaction network. 

\begin{definition}
\been
\item The complexes $ y$ and $ y'$ are {\em adjacent} if either $y = y'$ or $ y \ra  y' \in \RR$ or $ y' \ra  y \in \RR$. 
\item A {\em linkage class} $L$ is a subset of $\CC$ such that if $y \in L$ and $z  \in L$ then there exists a set $\{y =: y_1, y_2, \ldots, y_{n-1}, y_n := z \} \subset \CC$ such that $y_i$ is adjacent to $y_{i+1}$ for $1 \le i \le n-1$. We will use the notation $l$ for the number of linkage classes of a network. 
\item  A complex $ y$ is said to {\em react to} the complex $ y'$, denoted by $y \leadsto y'$, if either $y = y'$ or there exists a set $\{y =: y_1, y_2, \ldots, y_{n-1}, y_n := y' \} \subset \CC$ such that $\{y_1 \ra  y_2, y_2 \ra y_3, \ldots, y_{n-2} \ra y_{n-1}, y_{n-1} \ra  y_n\} \subset \RR$.  
\item A {\em strong linkage class} is a set of complexes $C \subset \CC$ such that for $y, z \in C$, $y \leadsto z$ and $z \leadsto y$. 

\item A {\em terminal strong linkage class} is a strong linkage class $C \subset \CC$ such that if $y \in C$ and $z \in \CC \setminus C$, then $y$ does not react to $z$. 

\enen
\end{definition}

\begin{definition}
 The {\em deficiency} of a chemical reaction network denoted by $\delta$ is defined to be $\delta=n-l-d$ where $n$ is the number of complexes in the network, $l$ is the number of linkage classes and $d$ is the dimension of the stoichiometric subspace. 
\end{definition}

The Deficiency Zero Theorem and the Deficiency One Theorem establish that a network with deficiency zero and a certain subclass of networks with deficiency one cannot admit MSS. 

\begin{theorem}[Deficiency Zero Theorem \cite{feinberg1972complex,horn1972necessary,HornJackson72}]\label{thm:defzero}
 Suppose that a chemical reaction network $N$ has deficiency $0$ and that each linkage class of $N$ is a terminal strong linkage class. For all positive reaction rate constants, the mass-action ODEs have precisely one steady state in each positive stoichiometric compatibility class. 
\end{theorem}


\begin{theorem}[Deficiency One Theorem \cite{FeinDefZeroOne}]\label{thm:defone}
Consider a chemical reaction network endowed with mass action kinetics, and with $l$ linkage classes, each containing just one terminal strong linkage class. Suppose that the deficiency of the network is $\de$, that the deficiencies of the individual linkage classes are $\de_j, j=1,\ldots,l$, and that these numbers satisfy the following conditions:
\been[(i)]
\item $\de_j \le 1, ~~1,\ldots,l$. 
\item $\sum_{j=1}^l \de_j = \de$.
\enen
Then, for an arbitrary choice of rate constants, the chemical reaction network does not admit multiple steady states within a positive stoichiometric compatibility class. 
\end{theorem}

The Deficiency One Algorithm stated in Section \ref{sec:defonealg} requires the following regularity condition on the network. Most networks arising as models of chemical processes satisfy this regularity condition. 

\begin{definition}
A network is considered to be {\em regular} if it satisfies the following conditions:
\been
\item \label{reg1} The reaction vectors of the network are positively dependent. In other words, there exists a set of positive numbers $\{ \alpha_{ y \ra  y'}|  y \ra  y' \in \RR\}$ such that $\sum_\RR \alpha_{ y \ra  y'} ( y' -  y) =  0$. 
\item \label{reg2} Each linkage class in the network contains just one terminal strong linkage class.
\item \label{reg3} For each pair of adjacent complexes $\{ y_i,  y_j\}$ in a terminal strong linkage class of the linkage class $L$, let $\RR_{\{ y_i,  y_j\}} := \{ y_i \to y_j, y_j \to y_i \} \cap \RR$. The number of linkage classes in the reaction network $\{\SSS|_{\CC|_{\RR \setminus \RR_{\{ y_i,  y_j\}}}},\CC|_{\RR \setminus \RR_{\{ y_i,  y_j\}}}, \RR \setminus \RR_{\{ y_i,  y_j\}}\}$ is strictly greater than the number of linkage classes in $\{\SSS,\CC,\RR\}$. (The linkage class $L$ disconnects when the reactions in $\RR_{\{ y_i,  y_j\}}$ are removed.)
\enen
\end{definition}

\section{A review of the Deficiency One Algorithm} \label{sec:defonealg}

Here we review the Deficiency One Algorithm of Feinberg \cite{FeinbergMSSdefone}. The Deficiency One Algorithm takes as input a regular deficiency one network with two or more linkage classes, each of deficiency zero, and determines whether the network permits multiple steady states or not. The algorithm has two variations depending on whether or not the network contains an irreversible reaction. We now describe the two algorithms. In the following we will let $s:= \abs{\SSS}$ denote the number of species, $r:= \abs{\RR}$ denote the number of reactions, and $n := \abs{\CC}$ the number of complexes. Let $\mu = (\mu_1,\ldots, \mu_{s}) \in \R^s$. 
\subsection{Deficiency one algorithm for a network that contains irreversible reactions} \label{sec:irrev}
\noindent {\bf Input:} A regular deficiency one network with two or more linkage classes, each of deficiency zero and such that there is at least one irreversible reaction in the network. 
\been[{\bf Step 1.}]
\item Determine a set of numbers, $\{g_1,g_2,\ldots,g_{n} \}$ not all zero, such that the following hold:
\been
\item $\sum_{i=1}^{n} g_i  y_i =  0$.
\item The $g_i$ corresponding to complexes in each linkage class sum to zero. In other words, for each linkage class $L$, $\sum_{i: y_i \in L} g_i =0$. 
\item The $g_i$ corresponding to complexes in each terminal strong linkage class $T$ sum to a nonnegative number. In other words, for each terminal strong linkage class $T$, $\sum_{i: y_i \in T} g_i \ge 0$. 
\enen

\item For a pair of adjacent complexes $ y_p$ and $ y_q$ in a terminal strong linkage class containing two or more complexes, remove the reaction arrows between the pair. Because of regularity Condition 3, the linkage class containing this terminal strong linkage class disconnects into two disjoint components. Sum over the $g_i$ associated with the complexes in one of the resulting two components of the linkage class. Write $ y_p \cdot {\mu} -  y_q \cdot {\mu} > 0$ (respectively $=0$, or $<0$) depending on whether the sum is positive (respectively is zero, or is negative). Repeat this step on the original network for every distinct pair of adjacent complexes in all terminal strong linkage classes.
\item Partition the set of reactant complexes in the network into three subsets $U, M$ and $L$ as follows:
\been
\item All complexes that do not belong to a terminal strong linkage class are placed in the subset $M$. 
\item All complexes in the same terminal strong linkage class are placed in the same subset. 
\enen
\item For the partition chosen in step 3, for each pair of distinct complexes $\{  y_i,  y_j\}$ in the subset $M$, write the relation $ y_i \cdot  \mu =  y_j \cdot  \mu$. 
\item For the partition chosen in step 3, do the following.
\been
\item For each complex $ y_i$ in $U$ and each complex $ y_j$ in $M$ write $ y_i \cdot  \mu >  y_j \cdot  \mu$. 
\item For each complex $ y_j$ in $M$ and each complex $ y_k$ in $L$ write $ y_j \cdot  \mu >  y_k \cdot  \mu$.
\item For each complex $ y_i$ in $U$ and each complex $ y_k$ in $L$ write $ y_i \cdot  \mu >  y_k \cdot  \mu$.
\enen
\item For the partition chosen in step 3, do the following.
\been
\item For each adjacent pair of complexes in each terminal strong linkage class contained in $U$, write the inequality from Step 2. 
\item For each adjacent pair of complexes in each terminal strong linkage class contained in $L$, write the inequality from Step 2 with the inequality sign reversed. 
\enen
\item Gather all the relations obtained in Steps 4-6 which results in the inequality system for the partition chosen in Step 3. 
\item Determine if there exists a nonzero vector $ \mu$ which satisfies the inequality system corresponding to the partition chosen and which is sign compatible with the stoichiometric subspace of the network. If there does exist such a vector then the network admits multiple positive steady states.
\item If Step 8 returns a nonzero vector $\mu$ then the algorithm is terminated. If not, then return to Step 3 and choose a new partition of the reactant complexes. Repeat Steps 4-8, and if necessary repeat this step. 
\enen
{\bf Output:} If there exists a nonzero vector $ \mu$ which is sign compatible with the stoichiometric subspace of the network and which satisfies the inequality system for some partition chosen in Step 3, then the network admits MSS. Otherwise, no matter what positive rate constants are chosen, the network does not admit MSS.

\subsection{Deficiency one algorithm for reversible networks}
\noindent {\bf Input:} A regular deficiency one network with two or more linkage classes, each of deficiency zero and such that all reactions in the network are reversible.\\
Carry out the same algorithm as the one for a network which contains irreversible reactions. If at the end of the algorithm, multistationarity is not established for the network, then repeat the algorithm with the signs of the $g_i$ chosen in Step 1 reversed. \\
{\bf Output:} Similar to Section \ref{sec:irrev}; establishes multistationarity or otherwise of a regular deficiency one reversible network. 

\section{Multistationarity in One-Reaction fully open networks}

Now we are ready to state our main theorem. The theorem stated below provides a complete characterization by multistationarity of one-reaction fully open networks. The proof of the theorem involves an application of Feinberg's Deficiency One Algorithm \cite{FeinbergMSSdefone}, along with the Deficiency Zero and Deficiency One theorems.

\begin{theorem} \label{thm:1rxn}
\begin{enumerate}
	\item Consider a fully open network endowed with mass action kinetics and which contains only one (irreversible) non-flow reaction:
	\begin{align*}
	 a_1 X_1 + a_2 X_2 + \cdots + a_s X_s \quad \ra \quad  b_1 X_1 + b_2 X_2 + \cdots + b_s X_s~,
	\end{align*}
	where $a_i,b_i\geq 0$.  
	Then the fully open network admits MSS if and only if the following holds: 
	\be \label{cond:irrev}
	\sum_{i:~b_i>a_i} a_i >1~.
	\ee
	Moreover, these multistationary fully open networks admit nondegenerate steady states.
	\item Consider a fully open network endowed with mass action kinetics which contains the following reversible non-flow reaction:
	\begin{align*}
	 a_1 X_1 + a_2 X_2 + \cdots + a_s X_s \quad \lra \quad  b_1 X_1 + b_2 X_2 + \cdots + b_s X_s~,
	\end{align*}
	where $a_i,b_i \geq 0$.  
	The fully open network admits MSS if and only if the following holds:
	\begin{align} \label{cond:rev}
	\sum_{i:~b_i>a_i} a_i >1 \quad {\rm or} \quad \sum_{i:~a_i>b_i} b_i >1~.
	\end{align}
	Moreover, these multistationary fully open networks admit nondegenerate steady states.
\end{enumerate}
\end{theorem}

\begin{remark}
Theorem \ref{thm:1rxn} establishes a relation between autocatalysis and multistationarity in one-reaction fully open networks. More precisely, conditions \eqref{cond:irrev} or \eqref{cond:rev} may be interpreted as follows: a one-reaction fully open network is multistationary if and only if it has an autocatalytic reaction and the sum of the stoichiometric coefficients of the autocatalytic species in the reactant complex is at least two. 
\end{remark}
Before we prove the theorem, we need a few technical lemmas. 
In the following two lemmas, we establish that certain simple but important one-reaction fully open networks admit nondegenerate MSS. 

\begin{lemma} \label{lemma:atom1}
Let $a_2>a_1>1$. Consider the following fully open network $N$ containing one non-flow reaction: 
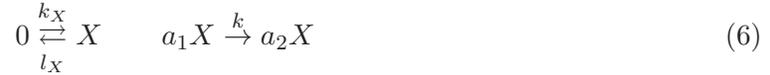
\begin{align} \label{atom1}
 	0  \stackrel[l_X]{k_X}{\rightleftarrows}  X \qquad a_1 X \stackrel[]{k}{\ra} a_2 X
\end{align}
Let $\ds k^* :=  \frac{1}{a_2-a_1}\left(\frac{l_X}{a_1}\right)^{a_1} \left(\frac{a_1-1}{k_X}\right)^{a_1-1}$. Then the following holds:

\begin{align*}
\text{If } k~ \begin{cases} \in (0,k^*) & \text{$N$ has two nondegenerate positive mass-action steady states} \\
= k^* & \text{$N$ has one doubly degenerate positive mass-action steady state}\\
> k^* & \text{$N$ has no positive mass-action steady states}. 
\end{cases}
\end{align*}

\end{lemma}

\begin{proof}
Let $x$ represent the concentration of the species $X$. The network \eqref{atom1} when endowed with mass-action kinetics results in the following dynamical system: 
\bd
\dot x = f(x) = k_X - l_X x + k(a_2-a_1)x^{a_1}
\ed
First note that by Descartes' rule of signs, since there are two sign changes in the coefficients of $f(x)$, the polynomial $f(x)$ has at most 2 positive roots (in fact, when counted with degeneracy $f(x)$ can have either 2 or 0 positive roots). 

Clearly $f(0)=k_X>0$ and $f(x) \ra \infty$ as $x \ra \infty$. Suppose a minimum of $f(x)$ occurs at $x^*$.  Since $f'(x^*)= -l_X + ka_1 (a_2-a_1){(x^*)}^{a_1-1} =0$, we have $x^* = \left(\frac{l_X}{ka_1(a_2-a_1)}\right)^{\frac{1}{a_1-1}}$ showing that there is a unique minimum and $f(x^*) = k_X + l_X \left(\frac{1}{a_1} - 1 \right) x^*$. If $f(x^*)<0 ~ (=0, >0)$ then there are two (one, zero resp.) positive steady states. Solving these inequalities for $k$ gives the desired condition. For the case where there are two steady states, say $x_1$ and $x_2 (\ne x_1)$, since $x_1 \ne x^*$ and $x_2 \ne x^*$, we have $f'(x_1) \ne 0$ and $f'(x_2) \ne 0$, so two steady states $x_1$ and $x_2$ are nondegenerate. When $k=k^*$, we have that $f(x^*) = f'(x^*) = 0$, so $x^*$ is the unique steady state and $x^*$ has degeneracy 2.  
\end{proof}

\begin{example} 
Letting $a_1=5, a_2=8, k_X=4,l_X=15$ in \eqref{atom1} results in the following network:
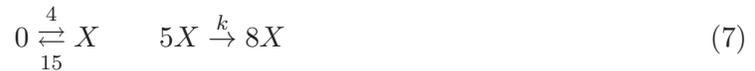
\begin{align} 
 	0  \stackrel[15]{4}{\rightleftarrows}  X \qquad 5 X \stackrel[]{k}{\ra} 8 X
\end{align}
for which
\bd
f(x) = 3kx^5 - 15x +4.
\ed
$f(x)$ attains the minimum at $x^* = \left(\frac{1}{k}\right)^{1/4}$. We calculate that $k^*=81$ and consider three cases where: 1) $k \in (0, k^*)$, 2) $k=k^*$, and 3) $k > k^*$. 
\been
\item For $k=50 < 81 = k^*$, $f(x) = 150x^5 - 15x +4$ has a pair of complex conjugate roots ($x_{\pm} \approx -0.0596238 \pm 0.578275 i$), a negative root ($x_1 \approx -0.615306$) and two positive roots ($x_2 \approx 0.285702$ and $x_3 \approx 0.448851$). 
\item For $k = 81 = k^*$, the Jacobian function is $f(x) = 243x^5 - 15x +4$ for which we find that $x^* = \left(\frac{l_X}{ka_1(a_2-a_1)}\right)^{\frac{1}{a_1-1}} = \frac{1}{3}$ is a root with degeneracy 2, and the other roots are either complex (with non-zero imaginary part), or are negative. 
\item For $k=100 > k^*$, $f(x) = 300x^5 - 15x +4$ has no positive real roots. 
\enen  
\end{example}

\begin{lemma} \label{lemma:atom2}
Let $b_1>1$ and $b_2>1$. The following fully open network $M$ containing one non-flow reaction admits MSS.
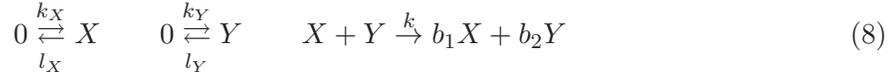
\begin{align} \label{atom2}
 	0  \stackrel[l_X]{k_X}{\rightleftarrows}  X \qquad 0  \stackrel[l_Y]{k_Y}{\rightleftarrows}  Y \qquad  X + Y \stackrel[]{k}{\ra} b_1 X + b_2 Y
\end{align}
Furthermore, $M$ has two positive nondegenerate mass-action steady states if and only if the parameters satisfy the following inequality:
\begin{align} \label{cond:atom2}
\frac{l_Y}{4k(b_1-1)l_X k_X} \left(l_X + \frac{k}{l_Y} (k_X(b_2-1) - k_Y (b_1-1)) \right)^2 > 1
\end{align}
In the case of equality in the above equation, $M$ has one doubly degenerate positive mass-action steady state, and in the case of the reverse inequality $M$ has no positive mass-action steady states. 
\end{lemma}

\begin{proof}
Let $x$ and $y$ represent the concentrations of the species $X$ and $Y$ respectively. The network \eqref{atom2} when endowed with mass-action kinetics results in the following dynamical system. 
\begin{align} \label{mideq0}
\left( \begin{array}{c} \dot x \\ \dot y \end{array} \right) = \left( \begin{array}{c} f_1(x,y) \\ f_2(x,y) \end{array} \right) = \left( \begin{array}{c} k_X - l_X x + k(b_1-1)x y \\ k_Y - l_Y y + k(b_2-1)x y \end{array} \right)
\end{align}
Some straightforward calculation reveals that the zeros of  \\
$f(x,y):=  (f_1(x,y) , f_2(x,y))$ coincide with the zeros of the following system
\begin{align} \label{mideq1}
y &= \frac{k_Y}{l_Y} - \frac{1}{l_Y} \left(\frac{b_2-1}{b_1-1}\right) (k_X - l_X x) \nonumber \\
g(x):=k_X& - \left(l_X + \frac{k}{l_Y} (k_X (b_2-1) - k_Y (b_1-1)) \right) x + k(b_1-1) \frac{l_X}{l_Y} x^2 = 0 
\end{align}
So that $N$ has two distinct positive steady states if and only if the second equation in \eqref{mideq1} has two distinct positive roots which occurs when $g(x^*)<0$ where $x^*$ is the minimum of $g(x)$.  The inequality $g(x^*) <0$ is easily shown to be equivalent to \eqref{cond:atom2}. It only remains to show that the set of positive parameters satisfying the inequality \eqref{cond:atom2} is non-empty. To this end, let $k_Y := k_X \left( \frac{b_2-1}{b_1-1}\right)$, $k := \frac{l_Y}{2(b_1-1) l_X}$, and $l_X = k_X+1$. With these choices, the left side of \eqref{cond:atom2} is $\frac{(k_X+1)^2}{2k_X}$ which is greater than $1$ for all positive $k_X$. The nondegeneracy of the two steady states is clear. When $g(x^*)=0$, $x^*$ is a steady state with degeneracy 2, since $g'(x^*)=0$. 
\end{proof}

\begin{example}
For the network in \eqref{atom2}, let 
\bd 
H=\frac{l_Y}{4k(b_1-1)l_X k_X} \left(l_X + \frac{k}{l_Y} (k_X(b_2-1) - k_Y (b_1-1)) \right)^2.
\ed
 Let $b_1=b_2=2$, $k_X=k_Y$ and $\frac{l_Y}{k} = 2 l_X$, so that $H = \frac{l_X^2}{2k_X}$ and using the definition in \eqref{mideq1}, $g(x) = k_X - l_X x + \frac{1}{2} x^2$. We consider three cases:
\been
\item $l_X = k_X=1$ results in the network  
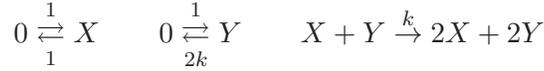
\begin{align*} 
 	0  \stackrel[1]{1}{\rightleftarrows}  X \qquad 0  \stackrel[2k]{1}{\rightleftarrows}  Y \qquad  X + Y \stackrel[]{k}{\ra} 2 X + 2 Y
\end{align*}
so that $g(x) = 1- x + \frac{1}{2}x^2$, which has no positive roots and thus the network has no positive steady states.
\item $l_X = k_X=2$ results in the network  
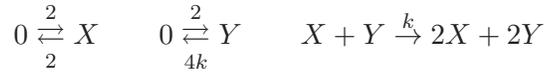
\begin{align*} 
 	0  \stackrel[2]{2}{\rightleftarrows}  X \qquad 0  \stackrel[4k]{2}{\rightleftarrows}  Y \qquad  X + Y \stackrel[]{k}{\ra} 2 X + 2 Y
\end{align*}
and $g(x) = 2-2x+\frac{1}{2}x^2$, which has the doubly degenerate root $x=2$. A simple calculation shows that $(x^*,y^*) = (2, 16k)$ is the unique steady state of the network with degeneracy 2.
\item $l_X =2,  k_X=1$ results in the network  
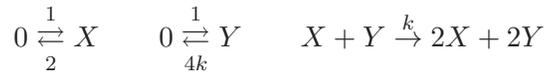
\begin{align*} 
 	0  \stackrel[2]{1}{\rightleftarrows}  X \qquad 0  \stackrel[4k]{1}{\rightleftarrows}  Y \qquad  X + Y \stackrel[]{k}{\ra} 2 X + 2 Y
\end{align*}
and $g(x) = 1-2x+\frac{1}{2}x^2$ which has the two positive roots $2 \pm \sqrt{2}$ resulting in the distinct nondegenerate steady states of the network, $(2-\sqrt{2}, 8k(2-\sqrt{2}))$ and $(2+\sqrt{2}, 8k(2+\sqrt{2}))$. 

\enen
\end{example}

\begin{remark}
The reaction networks studied in the Lemmas \ref{lemma:atom1} and \ref{lemma:atom2} are `one-reaction atoms of multistationarity' (see Definition \ref{def:atom}). The lemmas establish that these one-reaction atoms admit nondegenerate MSS and that the one-reaction atoms admit positive rate parameters for which there are two positive steady states. By a theorem of Joshi and Shiu (Lemma \ref{lemma:embed}), it follows that any fully open network which contains one of these atoms of multistationarity as an embedded network admits at least two positive nondegenerate mass-action steady states. 
\end{remark}

\begin{lemma} \label{lemma:ineqsystem}
Let $\{ a_1,a_2,\ldots,a_s,b_1,b_2,\ldots,b_s \}$ be a set of nonnegative integers. Consider the following system of inequalities:  
\begin{align} 
\sum_{i=1}^s a_i \mu_i &>  \max_{1 \le j \le s} \mu_j >0 \label{ineq1}\\
\text{sign} (\mu_i) & = \text{sign} (b_i - a_i) \quad \quad (1 \le i \le s). \label{ineq2}
\end{align}
This system has a solution $ \mu^* \in \R^{s} \setminus \{ 0 \}$ if and only if $\sum_{i: b_i > a_i} a_i >1$.
\end{lemma}

\begin{proof}
We first assume that the inequality system has a nonzero solution denoted by $ \mu^*$. So the set $\{j |\mu_j>0 \} = \{j |b_j>a_j \}$ is nonempty and 
\begin{align*}
   \sum_{i:b_i>a_i} a_i \mu^*_i & \ge 
     \sum_{i:b_i>a_i} a_i \mu^*_i +  \sum_{i:b_i<a_i} a_i  \mu^*_i  + \sum_{i:b_i=a_i} a_i \mu^*_i \nonumber \\
   &= \sum_{i=1}^s  a_i \mu^*_i > \max_{j: 1 \le j \le s} \mu^*_j = \max_{j: b_j > a_j} \mu^*_j > 0. 
\end{align*}
where in the last line we used \eqref{ineq1} twice. This shows that 
\begin{align}\label{ineq3}
 \sum_{i:b_i>a_i} a_i \mu^*_i > \max_{j: b_j > a_j} \mu^*_j 
 \end{align}
If $\sum_{i:b_i>a_i} a_i \le 1$, then there exists at most one $\widetilde i$ such that $b_{\widetilde i}>a_{\widetilde i}  > 0$.   So that $\sum_{i:b_i>a_i} a_i \mu^*_i \le \mu^*_{\widetilde i}$ which contradicts  \eqref{ineq3}. So we must have $\sum_{i:b_i>a_i} a_i > 1$. 

Conversely, assume that $\sum_{i:b_i>a_i} a_i > 1$. For all $i$ such that $b_i>a_i$, choose $\mu_i =1$, for all $k$ such that $b_k=a_k$, choose $\mu_k=0$, and for all $j$ such that $b_j < a_j$, choose $\mu_j = -\epsilon$ where $\epsilon>0$. This choice clearly satisfies \eqref{ineq2} and 
\begin{align*}
\sum_{i=1}^s a_i \mu_i &= \sum_{i:b_i>a_i} a_i \mu_i + \sum_{i:b_i<a_i} a_i \mu_i \\
&= \sum_{i:b_i>a_i} a_i - \epsilon \sum_{i:b_i<a_i} a_i \ge 2  - \epsilon \sum_{i:b_i<a_i} a_i > 1 = \max_{1 \le j \le s} \mu_j > 0
\end{align*}
where the last inequality follows by choosing $\epsilon$ sufficiently small. 
\end{proof}

%
A steady state $x_0$ of a system of ODEs, $\dot x = f(x)$ for $x \in \mathbb R^n$,  (or a steady state of a network which generates mass-action ODEs) is said to be {\em exponentially stable} if there is a neighborhood $V$ of $x_0$ and a positive constant $a$ such that $\abs{x(t)-x_0} < e^{-at}$ as $t \to \infty$ for all $x_0$ in $V$. 

If a fully open network $N$ is an embedded network of a fully open network $G$, then we can extend the steady states of $N$ to $G$ using the following result. 
\begin{lemma} \label{lemma:embed} [Theorem 4.2 and Corollary 4.6 of Joshi and Shiu \cite{joshi2012atoms}] Let $N$ be a fully open network embedded in a fully open network $G$. 
\beit
\item If $N$ admits nondegenerate MSS, then so does $G$. Moreover, if $N$ admits finitely many such steady states, then $G$ admits at least as many.
\item If $N$ admits multiple positive exponentially stable nondegenerate steady states then so does $G$. Moreover, if $N$ admits finitely many such steady states, then $G$ admits at least as many. 
\enit
\end{lemma}

The following lemma deals with some simple cases of one-reaction fully open networks  appearing in the statement of Theorem \ref{thm:1rxn}, which can be handled by using either Deficiency Zero Theorem or Deficiency One Theorem \cite{FeinDefZeroOne}. 

\begin{lemma} \label{lemma:defzeroone}
Consider the following one-reaction fully open network $N$:
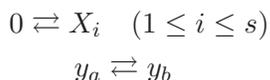
\begin{align*}
 0 \rla {X_i} & \quad (1 \le i \le s)\\
 y_a &\rla  y_b 
\end{align*}
where $ y_b \ra  y_a$ may have a zero reaction rate constant (in other words, the non-flow reaction is possibly irreversible). If $\{ y_a,  y_b \} \cap \{0,  X_1,\ldots, X_s \} \ne \emptyset$ then $N$ does not admit MSS. 

\end{lemma}

\begin{proof}
We will show that if $\{ y_a,  y_b \} \cap \{0,  X_1,\ldots, X_s \} \ne \emptyset$ holds then the network has a deficiency of either zero or one, so that either by applying Deficiency Zero or Deficiency One Theorem, multiple mass-action steady states can be ruled out. 
Since we are assuming that $ y_a \ra  y_b$ is a non-flow reaction, it is not the case that one of the complexes $\{ y_a,  y_b \}$ is the $0$ complex and the other is a unimolecular complex. So suppose that one of the complexes (either $ y_a$ or $ y_b$ but not both) in the non-flow reaction is either the $0$ complex or is unimolecular. Then the network has only one linkage class ($l=1$) and $n=s+2$ complexes so that the deficiency of the network is $\de = n-l-d =(s+2)-(1)-(s) = 1$. If the reaction $ y_b \ra  y_a$ has a positive rate constant, then the network is reversible and the entire unique linkage class is a terminal strong linkage class. Otherwise, the unique linkage class has exactly one terminal strong linkage class, either $ y_b$ or $\CC \setminus  y_a$, depending on whether $ y_a \in  \{0,  X_1,\ldots, X_s \}$ or $ y_b \in  \{0,  X_1,\ldots, X_s \}$, respectively. In either case, the hypotheses of the Deficiency One Theorem (Theorem \ref{thm:defone}) are satisfied which rules out multiple steady states for such a network.

On the other hand if both $ y_a$ and $ y_b$ are unimolecular, then $N$ has one linkage class ($l=1$), and $n=s+1$ complexes. So the deficiency of $N$ is $\de = n-l-d = (s+1)-(1)-(s)=0$. Whether the rate constant for the reaction $ y_b \ra  y_a$ is zero or positive, the entire unique linkage class is a terminal strong linkage class. So by the Deficiency Zero Theorem (Theorem \ref{thm:defzero}), $N$ does not admit MSS. 
\end{proof}

Now we are ready to prove the one-reaction fully open network theorem. 
\begin{proof}[Proof of Theorem \ref{thm:1rxn}]  We define $y_a := a_1X_1 + a_2X_2 + \ldots + a_sX_s$ and $y_b := b_1X_1 + b_2X_2 + \ldots + b_sX_s$. The network under study is the following fully open network:
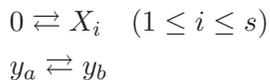
\begin{align*}
 & 0 \rla {X_i}  \quad (1 \le i \le s)\\
 & y_a \rla  y_b 
\end{align*}
where $ y_b \ra  y_a$ may have a zero reaction rate constant. 

Assume first that $\{ y_a,  y_b \} \cap \{0,  X_1,\ldots, X_s \} \ne \emptyset$; in other words, at least one of the complexes $ y_a$ or $ y_b$ is unimolecular or is the zero complex, which implies that 
$\sum_{i:~b_i>a_i} a_i \le 1$ and $\sum_{i:~a_i>b_i} b_i \le 1$. Furthermore by Lemma \ref{lemma:defzeroone}, the one-reaction network does not admit MSS. 

So it suffices to assume that $\{ y_a,  y_b \} \cap \{0,  X_1,\ldots, X_s \} = \emptyset$. From hereon, we will assume it to be the case that each of $ y_a$ and $ y_b$ is at least bimolecular. For the general one-reaction fully open network, 
\been[a.]
\item The set of species is $\SSS:=\{X_i | 1 \le i \le s\}$. 
\item The set of complexes is $\CC =   \{X_i | 1 \le i \le s\} \cup \{ 0\} \cup \{\sum_{i=1}^s  a_i X_i, \sum_{i=1}^s b_i X_i\}$. We relabel the complexes $ y_i := X_i$ for $1 \le i \le s$, $ y_{s+1} :=  0$,  $ y_{s+2} :=  y_a$ and $ y_{s+3} :=  y_b  $. 
\item The set of reactions is either $\RR_{irrev} = \{0 \rla X_i | 1 \le i \le s\} \cup \{ y_a \ra  y_b\}$ or $\RR_{rev} = \{0 \rla X_i | 1 \le i \le s\} \cup \{ y_a \ra  y_b\} \cup \{ y_b \ra  y_a\}$ depending on whether the non-flow reaction is irreversible or reversible. The two linkage classes of the network partition the set of complexes into $\CC_1:= \{X_i | 1 \le i \le s\} \cup \{ 0\} $ and $\CC_2 := \{\sum_{i=1}^s  a_i X_i, \sum_{i=1}^s b_i X_i\}$. 
\enen
The proof will proceed by application of the \doalgnospace. In 1 through 4 in the following, we lay the groundwork for application of the Deficiency One Algorithm by checking the conditions of validity of the algorithm. In 5, we apply the \doalg to the one-reaction fully open network  containing an irreversible non-flow reaction. In 6, we apply the \doalg to the one-reaction fully open network containing a reversible non-flow reaction, and finally in 7, we establish the nondegeneracy of the steady states for the multistationary fully open networks. 
\been
\item {\bf Dimension of the stoichiometric subspace.} For a fully open network, the stoichiometric subspace has `full' dimension. For $s$ species, the dimension of the stoichiometric subspace is $d=s$. 
\item {\bf Deficiency of the network.} 
Since we are assuming that neither of the two complexes in the non-flow reaction is the $0$ complex or is unimolecular, $\{ y_a,  y_b \} \cap  \{0,  X_1,\ldots, X_s \}  = \emptyset$. The number of complexes is $n=s+3$, the number of linkage classes is $l=2$ and the total deficiency is $\de = n-l-d = (s+3)-(2)-(s) = 1$. On the other, the deficiency of the two linkage classes is given by 
\been 
\item $\{0 \rla X_i | 1 \le i \le s\}$: $\de_1 = n_1-1-d_1 = (s+1)-(1)-(s)=0$. 
\item For $\{ y_a \rla  y_b\}$: $\de_2 = n_2-1-d_2 = (2)-(1)-(1)=0$. 
\enen
Such networks satisfy the hypotheses of the Deficiency One Algorithm \cite{FeinbergMSSdefone} provided they are also regular. In the following we check the regularity of the network. 

\item {\bf Regularity of the network.} 
\been 
\item {\em Regularity Condition 1:} The inflow reaction $0 \to X_j$ results in the canonical basis reaction vector $X_j$ whose $j$-th component is $1$ and all other components are $0$. On the other hand, the outflow reaction $X_j \to 0$ results in the reaction vector $- X_j$. The reaction vectors for the non-flow reactions are $v := (b_1-a_1, b_2-a_2, \ldots, b_2-a_s)$ and $(a_1-b_1,a_2-b_2, \ldots, a_s - b_s) = -v$. We will define a set of positive numbers $\{ \alpha_{ y \ra  y'}|  y \ra  y' \in \RR\}$ which satisfies $\sum_\RR \alpha_{ y \ra  y'} ( y' -  y) = 0$. 
\been[{Case} 1.]
\item  In the case where the non-flow reaction is reversible, we let $\alpha_{0 \to X_j}  = \alpha_{X_j \to 0} =1$ for all $j$ and we let $\alpha_{y_a \to y_b} = \alpha_{y_b \to y_a}= 1$. So, $\sum_\RR \alpha_{ y_i \ra  y_j} ( y_j -  y_i) =  \sum_j \left(X_j - X_j\right) + v - v = 0$, which shows that the reaction vectors are positively dependent. 
\item When the non-flow reaction is irreversible, we let $\alpha_{y_a \to y_b} =1$, and for all $j$ we let $\alpha_{0 \to X_j} = a_j +1 > 0 $ and $\alpha_{X_j \to 0} = b_j +1 > 0$. $\sum_\RR \alpha_{ y_i \ra  y_j} ( y_j -  y_i) = \sum_j \left( (a_j+1) X_j - (b_j+1) X_j  \right. \\ \left. + b_j X_j - a_j X_j \right)=0$.  
\enen
\item {\em Regularity Condition 2:} The entire linkage class $\{0,  X_1,\ldots, X_s \}$ is a terminal strong linkage class. Either the linkage class $\{ y_a,  y_b \}$ is a terminal strong linkage class or contains the terminal strong linkage class $\{ y_b\}$ depending on whether the non-flow reaction is reversible or not. 
\item {\em Regularity Condition 3:} If the inflow and outflow reactions of some species $j$ are removed or if the non-flow reaction were to be removed, the linkage class containing that reaction will be disconnected into two linkage classes. 
\enen
This shows that the network under consideration satisfies the hypotheses of the Deficiency One Algorithm.

\item {\bf Sign compatibility with the stoichiometric subspace.} Since the network under consideration is a fully open network, the stoichiometric subspace has ``full" dimension. In other words, the stoichiometric subspace $\mathbb S = \R^{s}$. Every vector $ \mu^* = (\mu_1, \ldots, \mu_s)$ is contained in $\mathbb S$ and therefore is trivially sign compatible with $\mathbb S$. Thus it suffices to determine an inequality system for a partition obtained in Step 3 of the algorithm and then to determine if the inequality system has any solution. If there is such a solution for any partition, then the network admits multiple steady states. 

\item {\bf Application of the Deficiency One Algorithm to a network containing at least one irreversible reaction.}
Suppose first that the non-flow reaction is not reversible. In other words $( y_b \ra  y_a) \notin \RR$. The set of reaction vectors for the network is $\RR_{irrev} = \{X_i, ~-X_i | i \in \SSS \} \cup \{\sum_{i=1}^s (b_i -a_i) ~ X_i\}$, so that $r = \abs{\RR} = 2s+1$.
\been[{\bf Step 1.}]
\item Let
\begin{align}
g_i &:= a_i-b_i, \qquad (1 \le i \le s) \nonumber \\
g_{s+1}&:= \sum_{i=1}^s (b_i-a_i), \quad g_{s+2}:=-1, \quad g_{s+3}:=1. \label{gvalues}
\end{align}
 We now check that the $g_i$ satisfy the three required conditions. 
\been[(a)]
\item  
$\ds \sum_{i=1}^{n} g_i  y_i = \sum_{i=1}^s (a_i-b_i) X_i + \left(\sum_{i=1}^s (b_i-a_i)\right)  0 + (-1) \sum_{i=1}^s a_i X_i + (+1) \sum_{i=1}^s b_i X_i =  0. $
\item It is clear that the $g_i$ sum to zero for the two linkage classes $\CC_1= \{X_i | 1 \le i \le s\} \cup \{  0\}$ and  $\CC_2 = \{\sum_{i=1}^s a_i X_i, \sum_{i=1}^s b_i X_i \}$. 
\item $\CC_1$ is a terminal strong linkage class while $\{\sum_{i=1}^s b_i X_i\}$ is a terminal strong linkage class within $\CC_2$. The corresponding $g_i$ sum to $0$ and $1$, respectively. 
\enen

\item $\CC_2$ contains only one complex in its terminal strong linkage class, and thus there are no pairs of complexes to consider for $\CC_2$. In $\CC_1$, all adjacent pairs of complexes are of the type $\{ 0, X_i\}$ for some $i$.  For a fixed $k$, let $ y_p := \{ X_k\}$ and $ y_q :=  \{0\}$.  So the sum over $g_i$ in $ y_p$ is  $g_k = a_k-b_k$ and $ y_p \cdot  \mu -  y_q \cdot  \mu =  X_k \cdot  \mu - 0= \mu_k$. So that for all $k$ we get the following system of inequalities:
\begin{align} \label{musign}
\text{sign}(\mu_i) = \text{sign}(a_i - b_i) \qquad (1 \le i \le s)
\end{align}
\item The set of reactant complexes is $\CC \setminus \{ y_b\}$. 
\been [(a)]
\item $ y_a$ is the only complex that does not belong to a terminal strong linkage class and so $ y_a \in M$. 
\item $\CC_1$ is a terminal strong linkage class and so all the complexes in $\CC_1$ must be placed in the same subset. The three choices of where to place the complexes in $\CC_1$ lead to the following three partitions of the reactant complexes.

\been[(i)]
\item $U = \emptyset, ~~ M = \{ y_a\} \cup \{ 0\} \cup \{X_i | 1 \le i \le s \}, ~~ L = \emptyset$.
\item $U = \emptyset, ~~ M = \{ y_a\}, ~~ L = \{ 0\} \cup \{X_i | 1 \le i \le s \}$.
\item $U = \{ 0\} \cup \{X_i | 1 \le i \le s \}, ~~ M = \{ y_a\}, ~~ L = \emptyset$.
\enen
\enen
It is straightforward to show that if the inequality system resulting after step 7 from partition (ii) has a solution $\mu^*$, then the inequality system resulting after step 7 from partition (iii) has a solution which is equal to $- \mu^*$. This is true whenever two partitions are related by switching the contents of the sets $U$ and $L$ \cite{FeinbergMSSdefone}. Thus, since partition (iii) does not provide any new information, we will restrict attention to partitions (i) and (ii). 
\item 
For partition (i), we get $\mu_i=0 ~ (1 \le i \le s)$. Since we are looking for a nonzero solution $ \mu^* = (\mu_1^*, \ldots, \mu_s^*)$, we may discard partition (i) and it suffices to consider partition (ii) only. Beginning here and in all the following steps, we will assume that we are considering partition (ii) even when this is not explicitly stated. For partition (ii), there is only one complex in $M$ and so we do not get any equations from applying this step.
\item Since there is one complex in $M$ and $s+1$ complexes in $L$, we get the following system of $s+1$ inequalities 
\begin{align*}
\sum_{i=1}^s a_i \mu_i > 0, ~~ \sum_{i=1}^s a_i \mu_i > \mu_j \quad (1 \le j \le s)  
\end{align*}
\item For each adjacent pair of complexes in $L$, we write the inequality from Step 2 with the inequality reversed. 
\begin{align*}
\text{sign} (\mu_i) = \text{sign} (b_i - a_i) ~~~ (1 \le i \le s)
\end{align*}
\item We gather all the inequalities from Steps 4-6 to get the inequality system. 
\begin{align*} 
\sum_{i=1}^s a_i \mu_i > 0, ~~ \sum_{i=1}^s a_i \mu_i > \mu_j \quad (1 \le j \le s), \\
\text{sign} (\mu_i) = \text{sign} (b_i - a_i) \quad (1 \le i \le s). 
\end{align*}
The first inequality holds only if there exists a $j$ such that $\mu_j>0$. So the first $s+1$ inequalities are equivalent to  $\sum_{i=1}^s a_i \mu_i >  \max_{1 \le j \le s} \mu_j >0$. The system of inequalities to be solved may be written as:
\begin{align} \label{ineqsystem}
\sum_{i=1}^s a_i \mu_i &>  \max_{1 \le j \le s} \mu_j >0 \nonumber\\
\text{sign} (\mu_i) & = \text{sign} (b_i - a_i), \quad \quad (1 \le i \le s)
\end{align}
\item By Lemma \ref{lemma:ineqsystem} the system of inequalities \eqref{ineqsystem} has a solution if and only if $\sum_{i: b_i > a_i} a_i >1$.
\item We have already shown in Steps 3 and 4 that we only need to consider partition (ii) since the other two partitions (i) and (iii) do not provide any new information. 
\enen
This completes the proof of the theorem for networks that contain at least one irreversible reaction.

\item {\bf Application of the Deficiency One Algorithm to a reversible network.} Now suppose that the non-flow reaction is reversible which results in a reversible network, in other words a network containing only reversible reactions. The set of reaction vectors for the network is $\RR_{rev} = \{X_i, ~-X_i | i \in \SSS \} \cup \{\sum_{i=1}^s (b_i -a_i) X_i\} \cup \{\sum_{i=1}^s (a_i -b_i) X_i \}$. 

The three networks $\RR_{rev}=\{0 \rla X_i | 1 \le i \le s\} \cup \{ y_a \ra  y_b\} \cup \{ y_b \ra  y_a\}$, $\RR_{irrev}=\{0 \rla X_i | 1 \le i \le s\} \cup \{ y_a \ra  y_b\} $, and $\widetilde{\RR}_{irrev}=\{0 \rla X_i | 1 \le i \le s\} \cup \{ y_b \ra  y_a\}$ generate the same stoichiometric subspace. By Lemma \ref{lemma:embed}, if either $\RR_{irrev}$ or $\widetilde{\RR}_{irrev}$ admits MSS then so does $\RR_{rev}$. This shows that if either $\sum_{i:b_i>a_i} a_i >1$ or $\sum_{i:a_i>b_i} b_i >1$, then $\RR_{rev}$ admits multiple steady states. 

Assume now that $\sum_{i:b_i>a_i} a_i \le 1$ and $\sum_{i:a_i>b_i} b_i \le 1$. It only remains to show that the network does not permit multiple steady states. Note that for reversible networks we need to go through the steps of the algorithm twice -- once for the $g_i$ chosen according to Step 1, and then again for $g_i$ which are negative of those chosen according to Step 1. 
\been [{\bf Step 1.}]
\item For $1 \le i \le s+2$, we let $g_i$ be the same as in \eqref{gvalues}. Since the complexes and the linkage classes are the same as in the irreversible case, there is nothing to check in parts (a) and (b). For part (c), we note that for a reversible network every terminal strong linkage class is a linkage class and so the $g_i$ corresponding to complexes in all terminal strong linkage classes sum to zero. 
\item We obtain the same set of inequalities as in \eqref{musign} for the terminal strong linkage class $\CC_1$. For the linkage class $\CC_2 = \{y_a, y_b\}$, we note that $1=g_{s+3} > g_{s+2} =-1$, so we get the following system of inequalities:
\begin{align*}
&\text{sign}(\mu_i) = \text{sign}(a_i - b_i) ~, ~~~~ 1 \le i \le s \\
&\sum_{i=1}^s (b_i - a_i) \mu_i >0
\end{align*}
\item The only constraint is that all complexes in $\CC_1$ belong to the same subset and all complexes in $\CC_2$ belong to the same subset. Thus there are $9$ distinct partitions. However, Remark 4.1.G in Feinberg \cite{FeinbergMSSdefone} tells us that we only need to examine partitions for which both subsets $U$ and $L$ are nonempty, since the condition that all terminal strong linkage classes contain more than one complex is satisfied. Furthermore, as in the irreversible case, interchanging contents of the subsets $U$ and $L$ result in equivalent inequality systems in the sense that if one inequality system has a solution then so does the other. Thus we only need to consider the following partition
$U = \CC_2 = \{ \sum_{i=1}^s a_i X_i,  \sum_{i=1}^s b_i X_i\},\quad M = \emptyset,\quad L = \CC_1 = \{ 0\} \cup \{X_i | 1 \le i \le s \}$. 
\item Since $M = \emptyset$, no equations result from this step.
\item Comparing complexes in $U$ and $L$ we get the following system
\begin{align*}
\sum_{i=1}^s a_i \mu_i > 0, ~~ \sum_{i=1}^s a_i \mu_i > \mu_j  \quad(1 \le j \le s) \\
\sum_{i=1}^s b_i \mu_i > 0, ~~ \sum_{i=1}^s b_i \mu_i > \mu_j \quad(1 \le j \le s).  
\end{align*}
\item We write the inequality from Step 2 with the inequality corresponding to the complex in $L$ reversed.
\begin{align*}
&\text{sign}(\mu_i) = \text{sign}(b_i - a_i) ~, ~~~~ 1 \le i \le s \\
&\sum_{i=1}^s (b_i - a_i) \mu_i >0
\end{align*}
\item We gather all the inequalities from Steps 4 to 6.
\begin{align} \label{ineqsystemrev}
\sum_{i=1}^s b_i \mu_i &> \sum_{i=1}^s a_i \mu_i >  \max_{1 \le j \le s} \mu_j >0 \nonumber\\
\text{sign} (\mu_i) & = \text{sign} (b_i - a_i), \quad \quad (1 \le i \le s)
\end{align}
\item By Lemma \ref{lemma:ineqsystem} this system of inequalities has no nonzero solutions for $\sum_{i:b_i>a_i} a_i \le 1$. 
\enen
As mentioned earlier, for the reversible case, we need to carry out the algorithm again with signs of all $g_i$ reversed. This results in an interchange of the roles of $a_i$ and $b_i$ and so it is straightforward to see that we get the following system of inequalities in Step 7.
\begin{align} 
\sum_{i=1}^s a_i \mu_i &> \sum_{i=1}^s b_i \mu_i >  \max_{1 \le j \le s} \mu_j >0 \nonumber\\
\text{sign} (\mu_i) & = \text{sign} (a_i - b_i), \quad \quad (1 \le i \le s)
\end{align}
Since by hypothesis $\sum_{i:a_i>b_i} b_i \le 1$, it follows from Lemma \ref{lemma:ineqsystem}, that the above system of inequalities does not have a nonzero solution. This completes the proof of the theorem for the reversible case.
\item {\bf Nondegeneracy of the steady states for the multistationary fully open networks.} We have shown that every multistationary one-reaction fully open network has an embedded fully open network of the form satisfying the hypotheses of either Lemma \ref{lemma:atom1} or of Lemma \ref{lemma:atom2}. Since the fully open networks appearing in Lemma \ref{lemma:atom1} and Lemma \ref{lemma:atom2} admit MSS, by Lemma \ref{lemma:embed}, it follows that if one of the  one-reaction fully open networks admits MSS, then it admits nondegenerate MSS. 
\enen
This completes the proof of the theorem.
\end{proof}

\begin{example}[Example \ref{ex:schlosser} continued] Theorem \ref{thm:1rxn} immediately helps classify the fully open networks M1-M3 in Example \ref{ex:schlosser} by multistationarity. 
\been
\item For network M1, in the notation of the proof of Theorem \ref{thm:1rxn} we let $y_a:= A+B$ and $y_b := 2A$, and we find that $\sum_{i: b_i > a_i} a_i =1$ and $\sum_{i: a_i > b_i} b_i =0$. So by Theorem \ref{thm:1rxn}, M1 does not admit MSS.
\item For network M2, $y_a:= 2A+B$ and $y_b := 3A$ and so $\sum_{i: b_i > a_i} a_i =2$ which implies that M2 admits nondegenerate MSS.
\item For network M3, $y_a:= A+2B$ and $y_b := 3A$ and so $\sum_{i: b_i > a_i} a_i =1$ and $\sum_{i: a_i > b_i} b_i =0$, thus M3 does not admit MSS.
\enen
\end{example}

We recall the following definition of CFSTR atom of multistationarity of Joshi and Shiu \cite{joshi2012atoms}. 

\begin{definition} \label{def:atom}
\been
\item A fully open network is a {\em CFSTR atom of multistationarity} if it admits nondegenerate MSS and it is minimal with respect to the embedded network relation among all such fully open networks.
\item A {\em one-reaction atom of multistationarity} is a CFSTR atom of multistationarity containing one non-flow (irreversible or reversible) reaction. 
\item A fully open network $G$ is said to {\em possess a CFSTR atom of multistationarity} if there exists an embedded network $N$ of $G$ that is a CFSTR atom of multistationarity.
\enen
\end{definition}

As a corollary of Theorem \ref{thm:1rxn}, we get a complete classification of all one-reaction atoms of multistationarity. We state this result as a theorem. 

\begin{theorem} \label{thm:1rxnatoms}
	
\been 
\item A one-reaction fully open network is a CFSTR atom of multistationarity if and only if it consists of one irreversible non-flow reaction and that non-flow reaction has one of the following two forms:
	\begin{align} \label{eqn:1rxnatom}
	a_1 X ~\ra~ a_2 X~, \quad {\rm or} \quad
	X+ Y ~\ra~ b_1 X + b_2 Y~, 
	\end{align}
	where $a_2>a_1>1$, or, respectively, $b_1>1$ and $b_2 >1$.  
\item A one-reaction fully open network possesses a CFSTR atom of multistationarity as an embedded network if and only if it admits nondegenerate MSS.
\enen
\end{theorem}
\begin{proof}
Evidently, the one-reaction fully open networks in \eqref{eqn:1rxnatom} admit nondegenerate MSS by Lemmas \ref{lemma:atom1} and \ref{lemma:atom2}. Furthermore, it is clear that both types of fully open networks are minimal in the class of multistationary fully open networks with respect to the embedded network relation. 

On the other hand, assume that $N$ is a one-reaction fully open network with multiple steady states. Then by Theorem \ref{thm:1rxn}, $N$ has a subnetwork containing a reaction of the following form:
\bd
c_1 X_1 + c_2 X_2 + \cdots + c_s X_s \quad \ra \quad  d_1 X_1 + d_2 X_2 + \cdots + d_s X_s
\ed
with $\sum_{i:d_i > c_i} c_i >1$. If there exists a $j$ such that $d_j>c_j>1$, then we let $X =X_j, ~  a_1=c_j$, and $a_2=d_j$. Otherwise, there exists a pair of indices $(i,j)$ such that $c_i=c_j=1$, $d_i>1$ and $d_j >1$. In that case, we let $(X,Y) = (X_i,X_j),~ b_1=d_i$, and $b_2=d_j$. 

Finally, the second part of the theorem follows from Lemma \ref{lemma:embed}. 
This completes the proof.
\end{proof}

The one-reaction atoms of multistationarity are useful beyond the one-reaction setting. If a one-reaction atom of multistationarity is an embedded network of a fully open network $N$ with possibly more than one non-flow reaction, then the `embedding theorem' of Joshi and Shiu (Lemma \ref{lemma:embed}), may be used to infer that $N$ has nondegenerate MSS. We state as a theorem the following important corollary of Theorem \ref{thm:1rxn}, whose scope of application is beyond the one-reaction setting.  

\begin{theorem} \label{thm:largenets}
A fully open network (which may contain more than one non-flow reaction) admits MSS if it contains as an embedded network a reaction of one of the following forms:
\been
\item $a_1 X ~\ra~ a_2 X$ with $a_2>a_1>1$.
\item $X+ Y ~\ra~ b_1 X + b_2 Y$ with $b_1>1$ and $b_2>1$. 
\enen
\end{theorem}

Recall from \cite{joshi2012atoms} that bimolecular networks are such that each complex in the network has at most two molecules. In other words, a complex in a bimolecular network has one of the following forms: $0$, $A$, $2A$ or $A+B$. Theorem \ref{thm:1rxn} establishes that there are no bimolecular one-reaction fully open networks with multiple steady states. This in turn implies that the smallest (by number of reactions) bimolecular fully open networks with multiple steady states should contain at least two non-flow reactions (where one reaction is not merely the reverse of the other reaction). In fact, there do exist bimolecular two-reaction CFSTR atoms of multistationarity and this set has been fully catalogued in \cite{joshi2012atoms}.

\begin{example}[Example \ref{ex:3nets} continued]
We can now answer the question posed in Example \ref{ex:3nets}. 
\been
\item For network $N1$ the fully open network $G1$ containing the non-flow reaction $2E \ra 3E$ is an embedded network. Since $G1$ is a one-reaction atom of multistationarity, it follows that $N1$ is multistationary. 
\item For the fully open network $N3$, first remove the reaction $A+E \ra 2E$, and then remove the species $C$ and $D$. This gives the fully open network $G3$ containing the non-flow reactions: $A \ra A+B$ and $2B \ra A$. $G3$ is one of the known two-reaction bimolecular atoms of multistationarity \cite{joshi2012atoms} and therefore $N3$ is multistationary.
\item  A straightforward calculation shows that $N2$ possesses no known atoms of multistationarity. In fact, plugging the fully open network $N2$ into the Chemical Reaction Network Toolbox \cite{Toolbox} reveals that $N2$ does not admit multiple positive steady states. 
\enen
\end{example}

We end by posing the question of identifying and characterizing `larger' CFSTR atoms of multistationarity, {\it i.e.} the ones that contain more than one non-flow reaction. Since tests that involve checking whether a certain large atom of multistationarity is embedded in a network may be computationally difficult, characterization of the CFSTR atoms of multistationarity in terms of general principles may be particularly helpful.

\subsection*{Acknowledgments}
This project was initiated by the author at a Mathematical Biosciences Institute (MBI) summer workshop under the guidance of Gheorghe Craciun. A conversation with Martin Feinberg was critical in directing the author in the right direction for completing the proof of the main theorem. The author is grateful to Anne Shiu for many great discussions and insightful comments. This manuscript has benefited greatly from the careful reading and detailed comments provided by the anonymous referees; the author thanks these referees for generously giving their time and attention.  
\bibliographystyle{elsarticle-num}
\bibliography{multistationarity}
\end{document}